\documentclass[12pt,a4paper]{amsart}
\usepackage{amssymb, enumerate, mdwlist, amsthm, amsmath, bm}
\usepackage[dvipdfmx]{xcolor,hyperref}

\hypersetup{
    colorlinks=true,
    citecolor=blue,
    linkcolor=red,
    urlcolor=orange,
}
\newcommand{\doi}[1]{doi:\href{https://doi.org/#1}{#1}}
\newtheorem{thm}{Theorem}
\newtheorem{cor}{Corollary}
\newtheorem{conj}{Conjecture}
\newtheorem{lemma}{Lemma}
\newtheorem{claim}{Claim}
\def\lognp{\tfrac{\log n}{-\log p}}
\def\E{\mathbb E}
\def\P{\mathbb P}
\makeatletter
\@namedef{subjclassname@2020}{%
    \textup{2020} Mathematics Subject Classification}
\makeatother

\title{Counting cliques in a random graph}
\author{Taro Sakurai}
\address{Chiba University\\Graduate School of Science\\Department of Mathematics and Informatics\\
1-33 Yayoi Inage\\Chiba\\263-8522 (JAPAN)\\
ORCID: 0000-0003-0608-1852}
\email{tsakurai@math.s.chiba-u.ac.jp}

\author{Norihide Tokushige}
\address{University of the Ryukyus\\College of Education\\
1 Senbaru Nishihara\\Okinawa\\903-0213 (JAPAN)\\
ORCID: 0000-0002-9487-7545}
\email{hide@edu.u-ryukyu.ac.jp}

\date{\today}
\begin{document}
\begin{abstract}
We show that the expected number of cliques in the Erd\H os--R\'enyi random 
graph $G(n,p)$ is $n^{\frac1{-2\log p}(\log n-2\log\log n+O(1))}$.
\end{abstract}
\maketitle
\section{Introduction}
In this note we estimate the expected number of cliques in a random graph. 
We start with some definitions.
Let $n$ be a positive integer, and let $0<p<1$ be a real number.
Let $G(n,p)$ denote the Erd\H os--R\'enyi random graph, that is, 
it has $n$ vertices and each of the $\binom n2$ possible edges occurs 
independently with probability $p$.
A clique is a maximal complete subgraph of a graph.
Let $X_{n,p}$ be a random variable counting cliques in $G(n,p)$. 
Our main result is the following.
\begin{thm}\label{thm:main}
\[
\E[X_{n,p}]=
n^{\frac1{-2\log p}(\log n-2\log\log n+O(1))}.
\]
\end{thm}
Using Markov's inequality Theorem~\ref{thm:main} has the 
following immediate consequence.
\begin{cor}
\[
 \lim_{n\to\infty}\P[X_{n,p}<n^{\frac{\log n}{-2\log p}}]=1.
\] 
\end{cor}

A $k$-clique is a clique with $k$ vertices.
The expected number of $k$-cliques in $G(n,p)$ is
\[
 F_n(k):= \binom nk p^{\binom k2}(1-p^k)^{n-k},
\]
and so $\E[X_{n,p}]=\sum_{k=1}^n F_n(k)$.
What is the most popular size $k$ which maximizes $F_n(k)$?
If this is given by $k=\tilde k$, then 
$F_n(\tilde k)\leq\E[X_{n,p}]\leq nF_n(\tilde k)$.
We show that $\tilde k$ is around $\lognp$, and 
$F_n(\tilde k)=n^{\frac1{-2\log p}(\log n-2\log\log n+O(1))}$.

Cliques (or independent sets) are one of the main objects in graph theory.
Moon and Moser \cite{MoonMoser} obtained the maximum number of cliques
in a graph. Bollob\'as and Erd\H os \cite{BollobasErdos} studied the 
maximum size of cliques, and the number of different sizes of
cliques in $G(n,p)$.
Our result is motivated by a recent result due to the first author \cite{Sakurai} 
concerning the number of maximal complete bipartite subgraphs of a 
random bipartite graph, which comes from a problem of counting 
formal concepts of random formal contexts in the theory of formal concept 
analysis originated by Wille \cite{Wille}. 
See also \cite{Bradley,Kovacs} for some application related to the number 
of cliques in $G(n,p)$.

The problem counting cliques in $G(n,p)$ naturally extends to the
hypergraph setting. Let $G^{(r)}(n,p)$ denote the random $r$-uniform 
hypergraph with $n$ vertices where each of the $\binom nr$ possible 
hyperedge occurs independently with probability $p$.
In this case a clique is a maximal complete subhypergraph. 
Let $X_{n,p}^{(r)}$ be a random variable counting cliques in $G^{(r)}(n,p)$. 
Then we have
\begin{align}\label{E[X^r]}
\E[X_{n,p}^{(r)}]=\sum_{k=1}^n 
\binom nk p^{\binom kr}(1-p^{\binom k{r-1}})^{n-k}.
\end{align}

\begin{conj}
 \[
\E[X_{n,p}^{(r)}]=\exp\left(
\left(\frac{\log n}{-\log p}\right)^{\frac1{r-1}}
\left(
(1-\tfrac1{r!})\log n-\tfrac1{r-1}\log\log n +O(1)
\right)
\right).
 \]
\end{conj}
By a routine computation of the term for 
$k=\lfloor\frac{\log n}{-\log p}\rfloor$ in
\eqref{E[X^r]} one can verify that the right hand side of the conjecture is indeed
a lower bound for $\E[X_{n,p}^{(r)}]$.
Theorem~\ref{thm:main} shows that the conjecture is true for $r=2$.

\section{Proof}
Let us define
\begin{align}\label{def:fn(x)}
f_n(x):=\log F_n(x)=\log\tbinom nx+\tbinom x2\log p+(n-x)\log(1-p^x),
\end{align}
and we will show that 
\[
 \max_{1\leq k\leq n}f_n(k)=\tfrac{\log n}{-2\log p}(\log n-2\log\log n+O(1)).
\]
\subsection{Lower bound}
\begin{lemma}\label{lemma:lower}
Let $x=\lognp+O(1)$. Then 
$f_n(x)=\frac{\log n}{-2\log p}(\log n-2\log\log n+O(1))$.
\end{lemma}
\begin{proof}
Let $c=\frac1{-\log p}$. Using Stirling's formula we have
\[
\binom nx =(1+o(1)) \sqrt{\frac n{2\pi x(n-x)}}\,
\frac{n^n}{x^x(n-x)^{n-x}}. 
\]
By the Taylor expansion we get
$\log(1-\tfrac{x}{n})=-\tfrac{c\log n}{n}+O(\tfrac1n)$.
Thus it follows that
\begin{align*}
\log\tbinom nx &= 
\tfrac12(\log n -\log 2\pi -\log x-\log(n-x))\\
&\qquad+n\log n - x\log x -(n-x)\log (n-x) + o(1)\\
&=x\log n - x\log x +O(\log n)\\
&=c(\log n)^2-c\log n\log\log n +O(\log n).
\end{align*}
We also have $(n-x)\log(1-p^{x})=-1+O(\tfrac{\log n}n)$. Finally we can rewrite
\eqref{def:fn(x)} to obtain
\begin{align*}
f_n(x) 
&=c(\log n)^2-c\log n\log\log n +\tfrac12(c\log n)^2\log p+O(\log n)\\
&=\tfrac{\log n}{-2\log p}\left(\log n-2\log\log n+O(1)\right),
\end{align*}
as needed.
\end{proof}

\subsection{Upper bound}
To give an upper bound for $f_n(x)$ from \eqref{def:fn(x)} we use
$\binom nx\le\frac{n^n}{x^x(n-x)^{n-x}}$, or equivalently,
$\log\tbinom nx\leq x\log n-x\log x-(n-x)\log(1-\tfrac xn)$
for $0<x<n$.
Then we have $f_n(x)\leq g_n(x)$, where
\[
 g_n(x):=x\log n-x\log x-(n-x)\log(1-\tfrac xn)+\tfrac{x(x-1)}2\log p.
\]

\begin{lemma}\label{lemma:upper}
For $1\leq x\leq n-1$ we have
$g_n(x)\leq\frac{\log n}{-2\log p}(\log n-2\log\log n+O(1))$.
\end{lemma}

We distinguish the cases $x\geq\frac{\log n}{-\log p}$ and
$x<\frac{\log n}{-\log p}$.

\begin{claim}
Lemma~\ref{lemma:upper} is true for $x\geq\lognp$.
\end{claim}
\begin{proof}
Let $a_n(x):=x\log n+\frac{x(x-1)}2\log p$.
Since $a_n''(x)=\log p<0$ it follows that $a_n'(x)$ is decreasing
in $x$, and if $a_n'(x)=0$ then $x=\lognp+\frac12$. Thus
\[
 a_n(x)\leq a_n(\lognp+\tfrac12)=\tfrac{(2\log n-\log p)^2}{-8\log p}
=\tfrac{(\log n)^2}{-2\log p}+O(\log n).
\]

Next let $b_n(x):=-x\log x-(n-x)\log(1-\frac xn)$.
Since $b_n'(x)=-\log x+\log(1-\frac xn)<0$ it follows that 
$b_n(x)$ is decreasing in $x$, and 
\begin{align*}
 b_n(x)&\leq b_n(\lognp)\\
&=\lognp\left(-\log\log n+\log(-\log p)+\log(1+\tfrac{\log n}{n\log p})\right)
-n\log(1+\tfrac{\log n}{n\log p})\\
%&=\lognp(-\log\log n+\log(-\log p)+o(1))+O(\log n)\\
&=\lognp(-\log\log n+O(1)).
\end{align*}
Then the result follows from $g_n(x)=a_n(x)+b_n(x)$.
\end{proof}

\begin{claim}
Lemma~\ref{lemma:upper} is true for $x<\lognp$.
\end{claim}
\begin{proof}
Let
\[
 h_n(x):=x\log n-x\log x+\tfrac{x(x-1)}2\log p.
\]
Then we have $g_n(x)=h_n(x)+O(\log n)$ because 
$|(n-x)\log(1-\frac xn)|\leq |n\log(1-\frac xn)|=O(\log n)$ 
for $x< \lognp$. To estimate $h_n(x)$ we have
\[
 h_n'(x)=\log n-\log x-1+x\log p-\tfrac12\log p,
\]
and
\[
 h''_n(x)=-\tfrac 1x+\log p<0.
\]
So $h_n'(x)$ is decreasing in $x$.
By solving $h_n'(x)=0$ we get
\[
x=\tfrac{W(\frac{-n\log p}{e\sqrt p})}{-\log p}
=\tfrac1{-\log p}(\log n-\log\log n+O(1))=:\tilde x,
\]
where $W$ is the Lambert $W$ function \cite{NIST}.
Thus $h_n(x)$ is maximized at $x=\tilde x$ and
the corresponding maximum value is
\[
h_n(\tilde x)=\tfrac{\log n}{-2\log p}\left(\log n-2\log\log n+O(1)\right),
\]
which gives the desired upper bound for $g_n(x)$.
\end{proof}

\subsection{Proof of Theorem~\ref{thm:main}}
\begin{proof}
By Lemmas~\ref{lemma:lower} and \ref{lemma:upper} we have
\[
\tilde f_n:=\max_{1\leq k\leq n}f_n(k)=\tfrac{\log n}{-2\log p}(\log n-2\log\log n+O(1)),
\]
and 
\[
 \E[X_{n,p}]=\exp (\tilde f_n+O(\log n))=
n^{\tfrac 1{-2\log p}(\log n-2\log\log n+O(1))}.
\qedhere
\] 
\end{proof}

\section*{Acknowledgment}
The second author was supported by JSPS KAKENHI Grant No. 18K03399.


\begin{thebibliography}{99}
\bibitem{BollobasErdos}
B.~Bollob\'as, P.~Erd\H os. 
Cliques in random graphs. 
Math.\ Proc.\ Cambridge Philos.\ Soc.\ 80 (1976), no. 3, 419--427,
\doi{10.1017/S0305004100053056}.

\bibitem{Bradley}
P.~E.~Bradley.
On the logistic behaviour of the topological ultrametricity of data. 
J.\ Classification 36 (2019), no. 2, 266--276,
\doi{10.1007/s00357-018-9281-y}.

\bibitem{Kovacs}
L.~Kov\'cs. 
Efficient approximation for counting of formal concepts generated from 
formal context. 
Miskolc Math.\ Notes 19 (2018), no. 2, 983--996,
\doi{10.18514/MMN.2018.2529}.

\bibitem{MoonMoser}
J.~W.~Moon, L.~Moser.
On cliques in graphs.
Israel J.\ Math.\ 3 (1965), 23--28,
\doi{10.1007/BF02760024}.

\bibitem{NIST}
F.~W.~J.~Olver, D.~W.~Lozier, R.~F.~Boisvert, C.~W.~Clark.
NIST Handbook of Mathematical Functions.
Cambridge University Press (2010).

\bibitem{Sakurai}
T.~Sakurai.
On formal concepts of random formal contexts. 
Inform.\ Sci.\ 578 (2021) 615--620,
\doi{10.1016/j.ins.2021.07.065}.

\bibitem{Wille}
R.~Wille.
Restructuring lattice theory: an approach based on hierarchies of concepts. 
Ordered Sets (Banff, Alta., 1981), pp. 445--470, NATO Adv. Study Inst. Ser. 
C: Math. Phys. Sci., 83, Reidel, Dordrecht-Boston, Mass., 1982,
\doi{10.1007/978-94-009-7798-3\_15}.
\end{thebibliography}
\end{document}